\documentclass{article}
\usepackage{spconf}

\usepackage{amsmath,amsfonts,amssymb,amsthm}

\usepackage{graphicx}
\usepackage{color}
\usepackage{array}
\usepackage{amsfonts}
\usepackage{multicol}
\usepackage{textcomp, gensymb}
\usepackage{mathtools}
\usepackage{graphicx}
\usepackage{subfigure}
\usepackage{bm}
\usepackage{comment}
\usepackage{xcolor}
\usepackage{arydshln}
\usepackage{tabularx}
\usepackage{multirow}
\usepackage{mathtools}

\newtheorem{theorem}{Theorem}
\newtheorem{lemma}{Lemma}
\newcommand{\E}{\mathbb{E}}
\newcommand{\drv}{\mathrm{d}}
\newcommand{\N}{\mathcal{N}}
\DeclareMathOperator*{\argmax}{arg\,max}
\DeclareMathOperator*{\argmin}{arg\,min}
\newcommand{\btheta}{\bm{\theta}}
\newcommand{\bpsi}{\bm{\psi}}
\newcommand{\R}{\mathbb{R}}

\allowdisplaybreaks

\ninept
\title{On Parametric Misspecified Bayesian Cram\'{e}r-Rao bound:\\ An application to linear Gaussian systems
\vspace{-0.3cm}
}
\name{
\vspace{-0.3cm}
Shuo Tang, Gerald LaMountain, Tales Imbiriba, Pau Closas\thanks{This work has been partially supported by the NSF under Award ECCS-1845833.}}
\address{Dept. of Electrical \& Computer Engineering, Northeastern University, Boston, MA (USA)}
\begin{document}
%
\maketitle
\begin{abstract}
A lower bound is an important tool for predicting the performance that an estimator can achieve under a particular statistical model. Bayesian bounds are a kind of such bounds which not only utilizes the observation statistics but also includes the prior model information. In reality, however, the true model generating the data is either unknown or simplified when deriving estimators, which motivates the works to derive estimation bounds under modeling mismatch situations. 
This paper provides a derivation of a Bayesian Cram\'{e}r-Rao bound under model misspecification, defining important concepts such as pseudotrue parameter that were not clearly identified in previous works. The general result is particularized in linear and Gaussian problems, where closed-forms are available and results are used to validate the results. 
\end{abstract}
\begin{keywords}
Model misspecification, Bayesian bound, Cram\'{e}r-Rao Bound, estimation theory.
\end{keywords}
%
\section{Introduction}
\label{sec:intro}

Parameter estimation is at the core of many statistical signal processing and machine learning disciplines, with applications in many practical problems such as positioning, navigation, wireless communication, image process, etc.
%
%
%
The study of unbiased estimators led to the derivation of different bounds on it's variance \cite{Fundamentals1993Kay,van2004detection}. Such bounds provide a benchmark to assess or compare the performance of estimators, as well as, characterizing the best attainable performance given a model and data distribution. These bounds can be classified into two categories: classical and Bayesian  \cite{closas2009bayesian}. Classical bounds include the well-known Cram\'{e}r-Rao bound (CRB), the Bhattacharyya Bound \cite{bhattacharyya1946some} and the the Barankin bound \cite{mcaulay1971barankin}. In the category of Bayesian bounds we can find the Bayesian Cram\'{e}r-Rao bound (BCRB) \cite{Fundamentals1993Kay, van2004detection,dauwels2005computing}, the Bayesian Bhattacharyya Bound \cite{bhattacharyya1946some}, the Weiss-Weinstein bound \cite{weiss1985lower} and the Ziv-Zakai bound \cite{Chazan1975}.


In this paper, we are interested in the Bayesian CRB, which assumes a general statistical model with unknown and random vector parameter $\bpsi = (\psi_1, \ldots, \psi_{n_\psi})^\top \!\!\!\in \R^{n_\psi \times 1}$, such that the model is characterized by its \textit{a priori} and likelihood distributions  
\begin{align}\label{eq:TrueModel}
    \mathcal{M}_{\ast} = \{\bm{x}|\bpsi \sim p_{\ast} (\bm{x} | \bpsi) ,\; \bpsi\sim p(\bpsi) ~:~ \bpsi \in \Psi \subset \mathbb{R}^{n_\psi}\}
\end{align}
\noindent respectively.
For any unbiased estimator of $\bpsi$, denoted as $\hat{\bpsi}(\bm{x})$, the BCRB states that 
\begin{align}\label{eq:BCRB}
    \E_{\bm{x},\bpsi}\big\{\big(\hat{\bpsi}(\bm{x}) - \bpsi\big)\big(\hat{\bpsi}(\bm{x}) - \bpsi\big)^\top\big\} - \bm{J}^{-1} \geq \mathbf{0} \;,
\end{align}
\noindent where the inequality involving the estimation error covariance denotes that the left-hand side is a positive semidefinite matrix. This result can be used to lower bound the terms in  the estimation error covariance $\E_{\bm{x},\bpsi}\big\{\big(\hat{\bpsi}(\bm{x}) - \bpsi\big)\big(\hat{\bpsi}(\bm{x}) - \bpsi\big)^\top\big\}$.
In \eqref{eq:BCRB}, $\bm{J} \in \R^{n_\psi \times n_\psi}$ denotes the so-called Bayesian Fisher Information Matrix (BFIM)
\begin{align}\label{eq:PFIM}
    \bm{J} &= \E_{\bm{x},\bpsi}\bigg\{\bigg(\frac{\partial}{\partial\bpsi} \ln p(\bm{x}|\bpsi)\bigg)\bigg(\frac{\partial}{\partial\bpsi} \ln p(\bm{x}|\bpsi)\bigg)^\top\bigg\} \nonumber\\
    &\qquad+ \E_{\bpsi}\bigg\{\bigg(\frac{\partial}{\partial\bpsi} \ln p(\bpsi)\bigg)\bigg(\frac{\partial}{\partial\bpsi} \ln p(\bpsi)\bigg)^\top\bigg\}\nonumber\\
    &= \bm{J}_D + \bm{J}_P \;,
\end{align}
\noindent which is composed of the Fisher Information Matrix ($\bm{J}_D$, accounting for the information on $\bpsi$ from the data) and the prior information matrix ($\bm{J}_P$, accounting for the a priori on $\bpsi$).

One major limitation of these theories is that these lower bounds are derived under the assumption that the true model \eqref{eq:TrueModel} which generated the observations is known. However, such knowledge of the true model is not always available \cite{zoubir2012robust}, either because the true underlying model is complex to characterize or because simplified models are preferred for the sake of tractability. 
In such situations the bounds mentioned above may become loose or even invalid since the estimators are derived based on a statistical model $\mathcal{M}$ that differs from the \emph{true} underlying model $\mathcal{M}_\ast$. In this context, model $\mathcal{M}$ is often referred to as \emph{misspecified} or \emph{mismatched} \cite{vuong1986cramer, richmond2015parameter}. 
Recently, there has been an increased interest in deriving bounds for the sort of mismatched problems described above. These works extend the theories of non-Bayesian bounds, mostly CRB-type, to the case of model misspecification where the most prominent works are in the area of Misspecified CRB (MCRB) \cite{fortunati2017performance}.

In that context, in a non-Bayesian setting where parameters are considered deterministic, an estimator designed optimally under $\mathcal{M}$ (e.g. the quasi maximum likelihood estimator, QMLE) is known to asymptotically converge to the so-called \textit{pseudotrue} parameter, which is the parameter in $\mathcal{M}$ that minimizes its Kullback-Leibler (KL) divergence with respect to $\mathcal{M}_\ast$.
As far as the authors know, Huber \cite{huber1967under} was the first at investigating the behavior of MLE under model misspecification  and discussed the concept of the pseudotrue parameter. Based on these ideas, the seminal work of Vuong \cite{vuong1986cramer} was the first to explore CRB-type lower bounds on estimation accuracy under model mismatches. This work used implicit function theory to provide the results and clarified the necessary assumptions for its derivation, which was inherited by many follow up works that appear only recently. 
For instance, Richmond and Horowitz \cite{richmond2015parameter} formalized the theory of MCRB by employing minimum norm theorem and the covariance inequality. Despite of the similar form of the bound, \cite{richmond2015parameter} resulted in a more restricted class of the suitable estimators than Vuong's original work \cite{vuong1986cramer} did due to the employed constraints on the minimum norm theorem. However, that work actually advanced the understanding of the result in \cite{vuong1986cramer} and provided a valid bound for any unbiased estimator of the pseudotrue. A detailed comparison of those two works can be found in \cite{richmond2018constraints}, which also provided great inspiration to this paper. Besides, the MCRB for parameter estimation has been recently used in many applications, such as sparse Bayesian learning \cite{richmond2016bayesian}, radar communications \cite{richmond2020misspecified}, or in the general class of elliptically distributed distributions \cite{fortunati2016misspecified}. In these works, the pseudotrue and the estimator are considered unknown but deterministic parameters and thus the bounds are used to bound estimators that are misspecified unbiased (MS-unbiased), such as the QMLE.

Less attention has been given to the study of estimation bounds under Bayesian mismatched models. In the review paper by Fortunati et al. \cite{fortunati2017performance}, this idea is discussed and deemed of relevance. 
In many situations, it is natural to account for the parameter's prior information in deriving an estimator, which adds an extra source of information but potentially adds another mismatch. Some works investigated the derivation of Bayesian estimation bounds under model mismatch, such as CRB-type \cite{kantor2015prior} or ZZB-type \cite{Gusi2014b,gusi2016mean}, although a general framework is still missing. 

In this paper, we provide a derivation of a Bayesian CRB for mismatched models, which we term as Misspecified BCRB (MBCRB). The derivation is done considering a parametric transformation of the true parameter space into the parameter space of the assumed model, similarly to the parametric approach in \cite{richmond2018constraints}. This provides a CRB-type result where the BCRB of the true parameters is mapped to the MBCRB of interest. In doing so, we propose a definition for the pseudotrue parameter, which is different from the one considered for MCRB as it accounts for the joint distribution of the assumed model. The bound is valid for MS-unbiased estimators with respect to such pseudotrue. 
The general MBCRB expression is particularized for linear and Gaussian models, where closed-form expressions are obtained. The bound is validated on synthetic data where various sorts of model mismatches are tested, showing tight prediction capabilities of the MBCRB.

\section{MISSPECIFIED BCRB}
\label{sec:MBCRB}
In this section we derive a parametric misspecified BCRB. 
For such, let us consider that $i)$ \eqref{eq:TrueModel} describes the true model, parameterized by $\bpsi$; and $ii)$ the assumed, potentially misspecified, model is given by
\begin{align}\label{eq:AssumedModel}
    \mathcal{M} = \{\bm{x}|\btheta \sim f (\bm{x} | \btheta) ,\; \btheta\sim f(\btheta)
~:~ \btheta \in \Theta \subset \mathbb{R}^{n_\theta}\}
\end{align}
\noindent where $\btheta = (\theta_1, \ldots , \theta_{n_\theta})^\top \in \R^{{n_\theta} \times 1}$ denotes the unknown random parameter which the estimator is attempting to infer from the available data $\bm{x}\in\mathbb{R}^{n_x \times 1}$ and the assumed statistical model $\mathcal{M}$. Notice that $\bpsi$ and $\btheta$ do not necessarily belong to the same parameter space, although they can in certain situations, as explained in \cite{fortunati2017performance}. 

We define the pseudotrue parameter as the parameter vector in $\Theta$ which minimizes the KL divergence between the true distribution and the assumed joint distribution $f(\bm{x},\btheta)$. The true model is set to be $p(\bm{x}|\psi)$ as Richmond and Fortunati did in \cite{richmond2015parameter, fortunati2017performance}, resulting in a pseudotrue that depends on the actual realization of $\bpsi$:
\begin{align}\label{eq:pseudotrue}
    \btheta_0(\bpsi) &= \argmin\limits_{\btheta} \mathcal{D}\bigg(p(\bm{x}|\bpsi)||f(\bm{x}, \btheta)\bigg)\nonumber\\
    &= \argmin\limits_{\btheta}\bigg( - \E_{\bm{x}|\bpsi}\bigg\{\ln f(\bm{x}, \btheta)\bigg\}\bigg),
\end{align}
which is slightly different to the pseudotrue parameter defined in other the MCRB works, \cite{vuong1986cramer,richmond2015parameter,richmond2018constraints,fortunati2017performance}. Based on this pseudotrue, this paper aims to find a lower bound for $\E_{\bm{x},\bpsi}\big\{\big(\hat{\btheta}(\bm{x}) - \btheta_{0}(\bpsi)\big)\big(\hat{\btheta}(\bm{x}) - \btheta_{0}(\bpsi)\big)^\top\big\}$ given that in this case there is a prior distribution on the assumed model parameter. Additionally, \eqref{eq:pseudotrue} is a convenient choice as it can be shown that relevant estimators such as the \textit{maximum a posterior} (MAP) are asymptotically achieving that quantity, thus being MS-unbiased with respect to $\btheta_0(\bpsi)$. 

Before being able to state the main MBCRB results, there are two assumptions that are imposed to the true model, which are common assumptions in the BCRB context. Additionally, an assumption on the class of estimators the bound is applicable to is also made. Namely, 
$i)$ the derivative of the log-likelihood is zero-mean, $\E_{\bm{x}|\bpsi}\bigg\{\frac{\partial}{\partial\bpsi}\ln p(\bm{x}|\bpsi)\bigg\} = \mathbf{0}$; 
$ii)$ the prior distribution for $\bpsi$ is such that $p(\psi_i = \psi_{i,\textrm{min}}) = p(\psi_i = \psi_{i,\textrm{max}})=0$, where $\bpsi \in \Psi =\Psi_1 \times \cdots \times \Psi_{n_\psi}$ with $\Psi_i \triangleq [\psi_{i,\textrm{min}}, \psi_{i,\textrm{max}}]$ being the value range for each $\psi_i$, $i \in \{1, 2, \dots, n_\psi\}$ and $\psi_{i,\textrm{min}}$ and $\psi_{i,\textrm{max}}$ are independent of $\bpsi$; $iii)$ the estimator is MS-unbiased with respect to the pseudotrue parameter, that is $\E_{\bm{x}|\bpsi}\big\{\hat{\btheta}(\bm{x})\big\} = \btheta_0(\bpsi)$.
Furthermore, the following Lemmas summarize two results used in proving Theorem \ref{thm:MBCRB}, which states the main MBCRB result.


\begin{lemma}\label{cly:3}
For $i \in \{1, 2, \cdots, n_\psi\}$ and $j \in \{1, 2, \cdots, n_\theta\}$
\begin{equation}
    \int_{\Psi} \hat{\theta}_j(\bm{x}) \frac{\partial}{\partial\psi_i} p({\bm{x}, \bpsi})\drv\bpsi = 0
\end{equation}
where $\hat{\btheta} = (\hat{\theta}_1, \ldots, \hat{\theta}_m)^\top \in \R^{n_\theta \times 1}$ is the estimator of $\btheta_0(\bpsi)$.
\end{lemma}
\begin{proof}
Note that
$\int_{\Psi}\frac{\partial}{\partial\psi_i} p({\bm{x}, \bpsi})\drv\bpsi = \frac{\partial}{\partial\psi_i} p(\bm{x}) = 0$. Thus,
$
    \int_{\Psi}\hat{\theta}_j(\bm{x})\frac{\partial}{\partial\psi_i} p({\bm{x}, \bpsi})\drv\bpsi
    =
    \hat{\theta}_j(\bm{x}) \int_{\Psi} \frac{\partial}{\partial\psi_i} p({\bm{x}, \bpsi})\drv\bpsi = 0
$.
\end{proof}

\begin{lemma}\label{cly:4}
Given $\btheta_0(\bpsi) = (\theta_{0,1}(\bpsi), \ldots , \theta_{0,n_\theta}(\bpsi))^\top \in \R^{{n_\theta} \times 1}$,
\begin{align}
    \int_{\Psi}\theta_{0,j}(\bpsi)\frac{\partial}{\partial\psi_i} p({\bm{x}, \bpsi})\drv\bpsi 
    = - \int_{\Psi}\frac{\partial\theta_{0,j}(\bpsi)}{\partial\psi_i}p(\bm{x},\bpsi)\drv\bpsi
\end{align}
\end{lemma}
\begin{proof}
We first define $\bpsi_{-i}$ as the vector containing all the elements in $\bpsi$ except for $\psi_{i}$, such that
{\small
\begin{equation}
    \int\theta_{0,j}(\bpsi)\frac{\partial}{\partial\psi_i} p({\bm{x}, \bpsi})\drv\bpsi =  \int\!\!\!\int\theta_{0,j}(\bpsi)\frac{\partial}{\partial\psi_i} p({\bm{x}, \bpsi}) \drv\psi_i\drv\bpsi_{-i}\nonumber
\end{equation}}
then, integrating by parts the integral over $\psi_i$, we obtain
{\small
\begin{align}
    &\int\!\!\!\int\theta_{0,j}(\bpsi)\frac{\partial}{\partial\psi_i} p({\bm{x}, \bpsi}) \drv\psi_i\drv\bpsi_{-i}\nonumber\\
    =& \int\bigg(\theta_{0,j}(\bpsi)p({\bm{x}, \bpsi})\bigg|_{\psi_i = \psi_{i,\textrm{min}}}^{\psi_i = \psi_{i,\textrm{max}}}
    - \int\frac{\partial\theta_{0,j}}{\partial\psi_i}p({\bm{x}, \bpsi})\drv\psi_i\bigg)\drv\bpsi_{-i}\nonumber\\
    = &- \int \frac{\partial\theta_{0,j}(\bpsi)}{\partial\psi_i}p({\bm{x}, \bpsi})\drv\bpsi,
\end{align}}

\noindent 
Notice that the term evaluating the joint distribution of $\bm{x}$ and $\bpsi$ at the boundaries of $\Psi_i$ is zero according to the second assumption made on a priori distribution $p(\bpsi)$, since $p({\bm{x}, \bpsi}) = p({\bm{x}|\bpsi})p(\bpsi_{-i}|\bpsi_{i})p(\bpsi_{i})$. 
\end{proof}

\begin{theorem}[Misspecified Bayesian CRB]\label{thm:MBCRB}
Given the true model $\mathcal{M}_\ast$ in \eqref{eq:TrueModel} parameterized by $\bpsi$ and the assumed model $\mathcal{M}$ in \eqref{eq:AssumedModel} parameterized by $\btheta$, the error covariance of any MS-unbiased estimator satisfies that
{\small
\begin{align}\label{eq:MBCRB}
    &\E_{\bm{x},\bpsi}\big\{\big(\hat{\btheta}(\bm{x}) - \btheta_{0}(\bpsi)\big)\big(\hat{\btheta}(\bm{x}) - \btheta_{0}(\bpsi)\big)^\top\big\}\nonumber\\
    &- \E_{\bpsi}\left\{\frac{\partial\btheta_0(\psi)}{\partial\bpsi}\right\}\bm{J}^{-1}\E_{\bpsi}\left\{\frac{\partial\btheta_0(\psi)}{\partial\bpsi}\right\}^\top \geq \mathbf{0} \;,
\end{align}}

\noindent where the inequality indicates positive semidefiniteness and $\bm{J}$ is the BFIM of $\bpsi$ in \eqref{eq:PFIM}. 
\end{theorem}


\begin{proof}

According to Lemmas \ref{cly:3} and \ref{cly:4}, we have that for the $j$-th element in $\Theta$
%
\begin{align}
    \int_{\Psi}\!\!\big(\hat{\theta}_j(\bm{x}) \!-\! \theta_{0,j}(\bpsi)\big) \frac{\partial}{\partial\psi_i} p(\bm{x},\bpsi)\drv\bpsi \!=\!\!\!\int_{\Psi}\!\!\frac{\partial\theta_{0,j}(\bpsi)}{\partial\bpsi}p(\bm{x},\bpsi)\drv\bpsi. \nonumber
\end{align}
which in vector case is such that
%
\begin{align}
\!\!\!\!\int_{\Psi}\!\!\big(\hat{\btheta}(\bm{x}) - \btheta_{0}(\bpsi)\big) \!\bigg(\!\frac{\partial}{\partial\bpsi} p(\bm{x},\bpsi)\!\!\bigg)^{\!\!\top}\!\!\!\drv\bpsi = \!\!\!\int_{\Psi}\!\frac{\partial\btheta_0(\bpsi)}{\partial\bpsi}p(\bm{x},\bpsi)\drv\bpsi \;. \nonumber
\end{align}

Assuming $\bm{x}$ is continuous in its domain,  integrating with respect to $\bm{x}$ gives
\begin{align}
    &\iint\big(\hat{\btheta}(\bm{x}) - \btheta_{0}(\bpsi)\big) \bigg(\frac{\partial}{\partial\bpsi} \ln p(\bm{x},\bpsi)\bigg)^\top p(\bm{x},\bpsi)\drv\bpsi\drv\bm{x}\nonumber\\
    =& \iint\frac{\partial\btheta_0(\bpsi)}{\partial\bpsi}p(\bm{x},\bpsi)\drv\bpsi\drv\bm{x},
\end{align}
where the integral limits are neglected for a cleaner notation.
We can multiply both side by two arbitrary vectors $\bm{a} \in \R^{n_\theta\times 1}$ and  $\bm{b} \in \R^{n_\psi\times 1}$ 
{\small
\begin{align}
    &\iint\bm{a}^\top\big(\hat{\btheta}(\bm{x}) - \btheta_{0}(\bpsi)\big) \bigg(\frac{\partial}{\partial\psi} \ln p(\bm{x},\bpsi)\bigg)^\top\bm{b} p(\bm{x},\bpsi)\drv\bpsi\drv\bm{x} \nonumber\\
    &= \bm{a}^\top\iint\frac{\partial\btheta_0(\bpsi)}{\partial\bpsi}p(\bm{x},\bpsi)\drv\bpsi\drv\bm{x}\bm{b}
\end{align}}
such that the above equation becomes a scalar identity and we can utilize the Cauchy-Schwarz inequality
{
\begin{align}
    &\iint\bm{a}^\top\big(\hat{\btheta}(\bm{x}) - \btheta_{0}(\bpsi)\big)\big(\hat{\btheta}(\bm{x}) - \btheta_{0}(\bpsi)\big)^\top\bm{a}p(\bm{x},\bpsi)\drv\bpsi\drv\bm{x}\cdot \nonumber\\
    &\iint\bm{b}^\top\bigg(\frac{\partial}{\partial\bpsi} \ln p(\bm{x},\bpsi)\bigg)\bigg(\frac{\partial}{\partial\bpsi} \ln p(\bm{x},\bpsi)\bigg)^\top\bm{b} p(\bm{x},\bpsi)\drv\bpsi\drv\bm{x} \nonumber\\
    \geq& \bigg(\bm{a}^\top\iint\frac{\partial\btheta_0(\bpsi)}{\partial\bpsi}p(\bm{x},\bpsi)\drv\bpsi\drv\bm{x}\bm{b}\bigg)^2.
\end{align}}
We can rewrite the expression identifying terms
\begin{align}\label{eq:CSineq3}
    &\bm{a}^\top\E_{\bm{x},\bpsi}\big\{\big(\hat{\btheta}(\bm{x}) - \btheta_{0}(\bpsi)\big)\big(\hat{\btheta}(\bm{x}) - \btheta_{0}(\bpsi)\big)^\top\big\}\bm{a}\cdot\bm{b}^\top\bm{J}\bm{b} \nonumber \\
    \geq& \bigg(\bm{a}^\top\E_{\bpsi}\big\{\frac{\partial\btheta_0(\psi)}{\partial\bpsi}\big\}\bm{b}\bigg)^2,
\end{align}
where
\begin{equation}\label{eq:PFIM2}
    \bm{J} = \E_{\bm{x},\bpsi}\bigg\{\bigg(\frac{\partial}{\partial\bpsi} \ln p({\bm{x}, \bpsi})\bigg)\bigg(\frac{\partial}{\partial\bpsi} \ln p({\bm{x}, \bpsi})\bigg)^\top\bigg\}\nonumber\\
\end{equation}
is exactly the BFIM defined in \eqref{eq:PFIM}, after making use of the first regularity condition $\E_{\bm{x}|\bpsi}\bigg\{\frac{\partial}{\partial\bpsi}\ln p(\bm{x}|\bpsi)\bigg\} = \mathbf{0}$ such that there is no cross-terms between likelihood and a priori distributions. 

Since $\bm{a}$ and $\bm{b}$ are arbitrary, we consider
 \begin{equation}
     \bm{b} = \bm{J}^{-1}\E_{\bpsi}\left\{\frac{\partial\btheta_0(\bpsi)}{\partial\bpsi}\right\}^\top\bm{a}
 \end{equation}
and substitute it in \eqref{eq:CSineq3}
\begin{align}
    &\bm{a}^\top\E_{\bm{x},\bpsi}\big\{\big(\hat{\btheta}(\bm{x}) - \btheta_{0}(\bpsi)\big)\big(\hat{\btheta}(\bm{x}) - \btheta_{0}(\bpsi)\big)^\top\big\}\bm{a}\nonumber\\
    &\times \bm{a}^\top \E_{\bpsi}\left\{\frac{\partial\btheta_0(\bpsi)}{\partial\bpsi}\right\}\bm{J}^{-1} \E_{\bm{x},\bpsi}\left\{\frac{\partial\btheta_0(\bpsi)}{\partial\bpsi}\right\}^\top\bm{a}\nonumber\\
    \geq& \bigg(\bm{a}^\top\E_{\bpsi}\left\{\frac{\partial\btheta_0(\bpsi)}{\partial\bpsi}\right\}\bm{J}^{-1}\E_{\bpsi}\left\{\frac{\partial\btheta_0(\bpsi)}{\partial\bpsi}\right\}^\top\bm{a}\bigg)^2 \;
\end{align}
to finally obtain
{\small
\begin{align}\label{eq:MBCRB0}
    & \bm{a}^\top\E_{\bm{x},\bpsi}\big\{\big(\hat{\btheta}(\bm{x}) - \btheta_{0}(\bpsi)\big)\big(\hat{\btheta}(\bm{x}) - \btheta_{0}(\bpsi)\big)^\top\big\}\bm{a}\nonumber\\
    & \geq \bm{a}^\top\E_{\bpsi}\left\{\frac{\partial\btheta_0(\bpsi)}{\partial\bpsi}\right\}\bm{J}^{-1}\E_{\bpsi}\left\{\frac{\partial\btheta_0(\bpsi)}{\partial\bpsi}\right\}^\top\bm{a} \;,
\end{align}}

\noindent from which the main result of the Theorem \ref{thm:MBCRB} follows.   
\end{proof}


The MBCRB in \eqref{eq:MBCRB} appears in a sandwich form consisting of the BFIM inverse in the middle and two terms on the both sides involving the parametric transformation from $\Psi$ to $\Theta$ through the pseudotrue parameter, which is similar to the MCRB results \cite{richmond2015parameter, fortunati2017performance, richmond2018constraints}. Notice that $i)$ the MBCRB is accounting for the prior on $\bpsi$ and $\btheta$ through the expectations in \eqref{eq:MBCRB} but also implicitly via $\btheta_{0}(\bpsi)$; and $ii)$ when the model is correctly specified the pseudotrue coincides with $\btheta_{0}(\bpsi) = \bpsi$, and $\textrm{MBCRB}(\btheta=\bpsi) = \textrm{BCRB}(\bpsi) = \bm{J}^{-1}$. 

\section{APPLICATION TO linear and Gaussian SYSTEMS}
\label{sec:lineargaussian}
In this section we particularize the MBCRB result to bound the special class of mismatched linear and Gaussian systems, which are very common and serve as the foundation of complex system. More precisely, the true model is generally described  by  
{
\begin{align}
    \bpsi \sim \N(\bm{\mu}_{\bpsi}, \bm{\Sigma}_{\bpsi}), \,\,\, \bm{x}_n|\bpsi \sim \N(\bm{H}_*\bpsi, \bm{\Sigma}_*),\,\,\, n = 1, \ldots, N \nonumber
\end{align}
}

\noindent which generates the observations and from which $N$ samples are available. The true model is parameterized by $\{\bm{\mu}_{\bpsi}, \bm{\Sigma}_{\bpsi}, \bm{H}_* , \bm{\Sigma}_* \}$. 
Conversely, the assumed model is
\begin{equation}
    \btheta \sim \N(\bm{\mu}_{\btheta}, \bm{\Sigma}_{\btheta}), \,\,
    \bm{x}_n|\btheta \sim \N(\bm{H}\btheta, \bm{\Sigma}), \,\, n = 1, \ldots, N \nonumber
\end{equation}
\noindent where we notice that $\btheta$ may not be in the same parameter space as $\bpsi$, and that the model is parameterized by $\{\bm{\mu}_{\btheta}, \bm{\Sigma}_{\btheta}, \bm{H} , \bm{\Sigma} \}$. It can be shown that the value of $\btheta$ minimizing the KL divergence between $\N(\bm{H}_*\bpsi, \bm{\Sigma}_*)$ and $\N(\bm{H}\btheta, \bm{\Sigma}) \N(\bm{\mu}_{\btheta}, \bm{\Sigma}_{\btheta})$ is 
{\footnotesize
\begin{align}\label{eq:pseudotrueEg}
    \btheta_0(\bpsi) \!\!&=\!\! {\bigg(\!\!N\bm{H}^\top\bm{\Sigma}^{-1}\bm{H} + \bm{\Sigma}_{\btheta}\!\!^{-1}\!\!\bigg)\!\!}^{-1}\!\! \bigg(\!\!N\bm{H}^\top\bm{\Sigma}^{-1}\bm{H}_*\bpsi + \bm{\Sigma}_{\btheta}^{-1}\bm{\mu}_{\btheta}\!\!\bigg)\!,
\end{align}}
\noindent which is the pseudotrue parameter as defined in \eqref{eq:pseudotrue}.

Based on the assumed model and the observation matrix $\bm{X} = (\bm{x}_1, \dots, \bm{x}_N )$, we can derive \textit{optimal} estimators of $\btheta$. For instance, the MAP estimator becomes
{\footnotesize
\begin{align}
    \hat{\btheta}(\bm{X}) &= \argmax_{\btheta}\bigg(\ln p(\bm{X}|\btheta) + \ln p(\btheta)\bigg) \\
    &= \!\!{\bigg(N\bm{H}^\top\bm{\Sigma}^{-1}\bm{H} + \bm{\Sigma}_{\btheta}^{-1}\bigg)\!\!}^{-1} \bigg(\sum\limits_{n=1}^N\bm{H}^\top\bm{\Sigma}^{-1}\bm{x}_n + \bm{\Sigma}_{\btheta}^{-1}\bm{\mu}_{\btheta}\bigg) \;, \nonumber
\end{align}}

\noindent which is indeed MS-unbiased since $\E_{\bm{x}|\bpsi}\big\{\hat{\btheta}(\bm{X})\big\} = \btheta_0(\bpsi)$. To compute the MBCRB for this class of mismatched models we follow the result in Theorem \ref{thm:MBCRB}
\begin{align}
    \E_{\bm{x},\bpsi}\bigg\{\!\!\big(\hat{\btheta}(\bm{X}) \!-\! \btheta_0(\bpsi)\big){\big(\hat{\btheta}(\bm{X}) \!-\! \btheta_0(\bpsi)\big)\!\!}^\top\!\bigg\} \geq\bm{A}\bm{J}^{-1}\!\bm{A}^\top
\end{align}
\noindent with $\bm{A}$ defined as
%
{
\begin{align}
    \E_{\bpsi}\bigg\{\!\!\frac{\partial\btheta_0(\bpsi)}{\partial\bpsi}\!\!\bigg\} \!&=\! \bigg(\!\!N\bm{H}^\top\bm{\Sigma}^{-1}\bm{H} + \bm{\Sigma}_{\btheta}^{\!\!-1}\!\!\bigg)^{-1}\!\!\!\!N\bm{H}^\top\bm{\Sigma}^{-1}\bm{H}_* \!\!\doteq\!\! \bm{A}
\end{align}}
and the BFIM for $\bpsi$ computed as $\bm{J} = N\bm{H}_*^\top\bm{\Sigma}_*^{-1}\bm{H}_* + \bm{\Sigma}_{\bpsi}^{-1}$.
 
%

\section{SIMULATION RESULTS}
\label{sec:simulation}

We validated the MBCRB results in the linear and Gaussian context described in Section \ref{sec:lineargaussian}. The particular values of the true model are such that $n_\psi=3$, with the a priori mean and covariances being respectively
$\bm{\mu}_{\bpsi} = [10,\,20,\, 5]^\top$ and $\bm{\Sigma}_{\bpsi} = \sigma_{\bpsi}^2\bm{I} = 0.5\bm{I}$, where $\bm{I}$ denotes the identity matrix with the corresponding dimension. The observation model is set to $\bm{H} = h_*\bm{I}$ and $h_* = 1$. The covariance of the observed data is $\bm{\Sigma}_* = \sigma_*^2\bm{Q}$, with $\sigma_*^2 = 0.04$ and $\bm{Q}_{i,j} = \rho^{|i-j|}$, following an order-1 auto-regressive (AR-1) signal model \cite{Imbiriba_echo_2018}, controlled by a correlation parameter $\rho=0.5$,
leading to a signal-to-noise ratio of $\textrm{SNR} = 34$ dB.
The specific values of the parameters in the assumed model are discussed for the different scenarios.

\subsection{Comparison of MBCRB and BCRB}
In this experiment, the parameters of the misspecified assumed model are such that $\bm{\mu}_{\btheta} = [8,\,18,\, 6]^\top$ and $\bm{\Sigma} = \sigma^2\bm{I} = 0.1\bm{I}$. The rest of the assumed parameters coincide with those of the true model. The Monte Carlo simulations are averaging $10^5$ independent realizations, where each time $N$ samples are be observed. For the sake of clarity, the results in Fig. \ref{fig:RMSEvsMBCRBvsBCRB} depict the root mean square error (RMSE) and theoretical bound of the first element in the parameter vector, that is $\theta_1$, for different numbers of samples $N$. Although the usual BCRB fails to lower-bound the variance of the estimation error when low number of samples are available (i.e., the prior plays a bigger role), the MBCRB can lower-bound it tightly.

\begin{figure}[htb]
\vspace{-0.2cm}
  \centering
  \centerline{\includegraphics[width=0.4\textwidth]{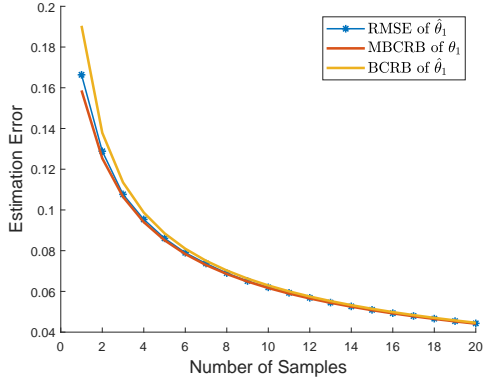}}
 \vspace{-0.3cm}
\caption{RMSE vs MBCRB and BCRB as a function of $N$.}
\label{fig:RMSEvsMBCRBvsBCRB}
\vspace{-0.5cm}
\end{figure}

\subsection{Extended Biased Bound}
Given that in the proposed setup both true and assumed models were conveniently chosen such that 
$\btheta$ and $\bpsi$ belong to the same vector space $\Theta=\Psi$, the bound can be extended to bound the covariance of $\hat{\btheta} - \bpsi$. Similar to what's discussed in \cite{richmond2015parameter, fortunati2016misspecified, fortunati2017performance}, such bound involves the MBCRB and an additional term:
%
\begin{align}\label{eq:biasedBound}
     &\E_{\bm{x},\bpsi}\big\{\big(\hat{\btheta}(\bm{x}) - \bpsi\big)\big(\hat{\btheta}(\bm{x}) - \bpsi\big)^\top\big\}\nonumber\\
    &\geq \E_{\bm{x},\bpsi}\left\{\frac{\partial\btheta_0(\bpsi)}{\partial\bpsi}\right\}\bm{J}^{-1}\E_{\bm{x},\bpsi}\left\{\frac{\partial\btheta_0(\bpsi)}{\partial\bpsi}\right\}^\top + \bm{r}\bm{r}^\top,
\end{align}
where the biased term is $\bm{r} = \btheta_0(\bpsi) - \bpsi$.
Under the same parameter settings of the previous experiment except for that the number of samples $N$ is fixed to $40$, Fig. \ref{fig:RMSEvsbiasedBound} shows that the error between the estimator and the true parameter can be lower-bounded by the extended biased bound in \eqref{eq:biasedBound}.
\begin{figure}[htb]
  \centering
  \centerline{\includegraphics[width=0.4\textwidth]{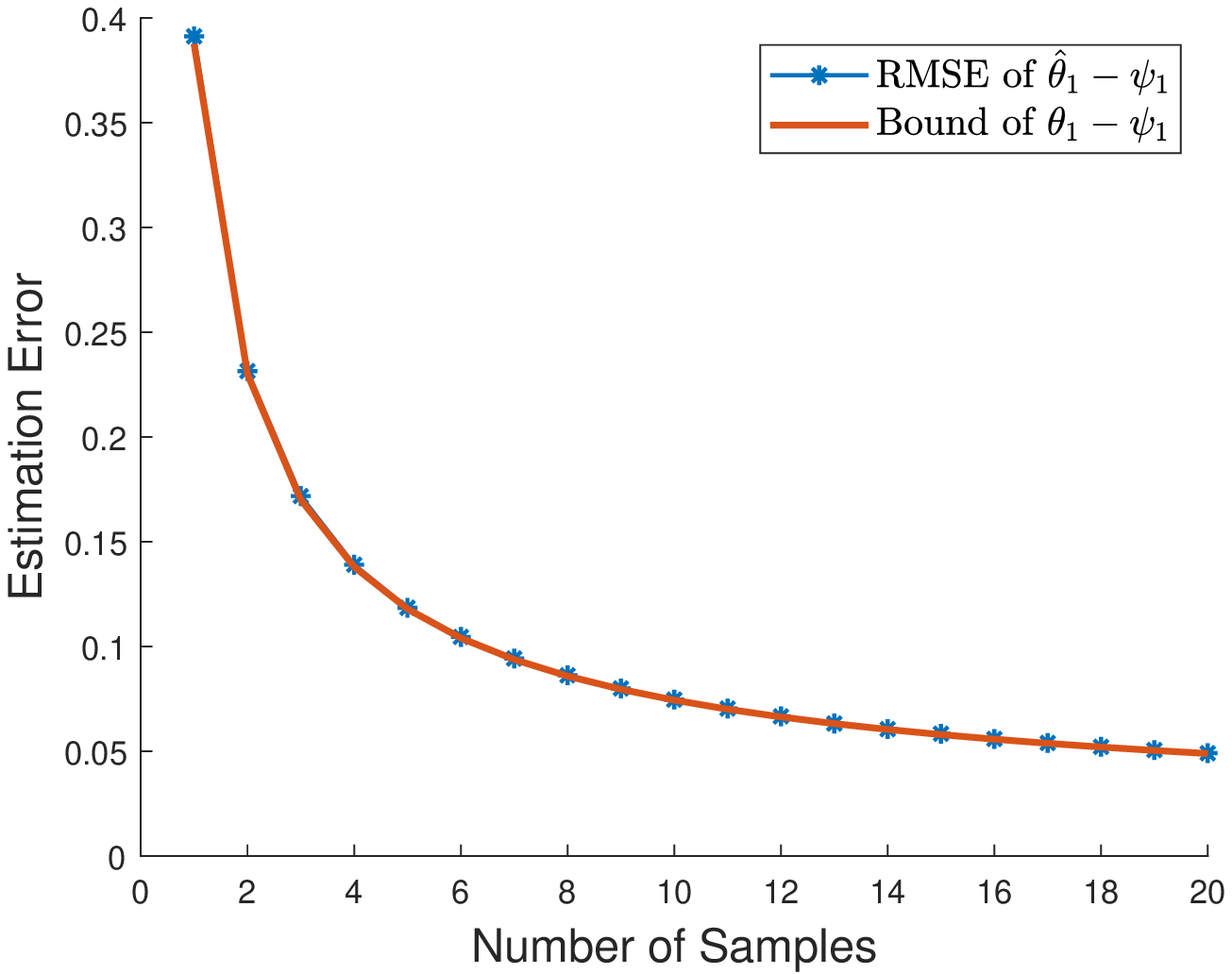}}
 \vspace{-0.3cm}
\caption{RMSE vs bound on $\hat{\btheta} - \bpsi$ as a function of $N$.}
\label{fig:RMSEvsbiasedBound}
\vspace{-0.25cm}
\end{figure}

\subsection{Different Levels of model misspecification}
A set of experiments are done to compare the RMSE and the bound on $\hat{\btheta} - \bpsi$ under different levels of model misspecification. For compatibility, we leverage the same true and assumed models as described at the beginning of the section except for the following mismatched parameters.
The simulation shown in Fig. \ref{fig:RMSEvsbiasedBound_different_level} (top) is implemented with number of samples $N = 50$ and the misspecified parameter is $\bm{H} = h\bm{I}$ for several values of $h$. The RMSE reaches the minimum value when $h = h_* = 1$, as expected. Similarly, Fig.\ref{fig:RMSEvsbiasedBound_different_level} (bottom) shows how the RMSE and bound fluctuate with the varying covariance of the observation model $\bm{\Sigma} = \sigma^2\bm{I}$. The simulations are implemented with $N = 500$ to get smoother curves. 
Notice that since $\bm{\Sigma}_* = \sigma_*^2\bm{Q}$ and there is no value of $\btheta$ that would make the assumed model to be exactly the same as the true model. However, the bound is able to predict the RMSE for various levels of mismatches in $\bm{\Sigma}$ compared to $\bm{\Sigma}_*$.   



\begin{figure}[t]
    \centering
    \includegraphics[width=0.4\textwidth]{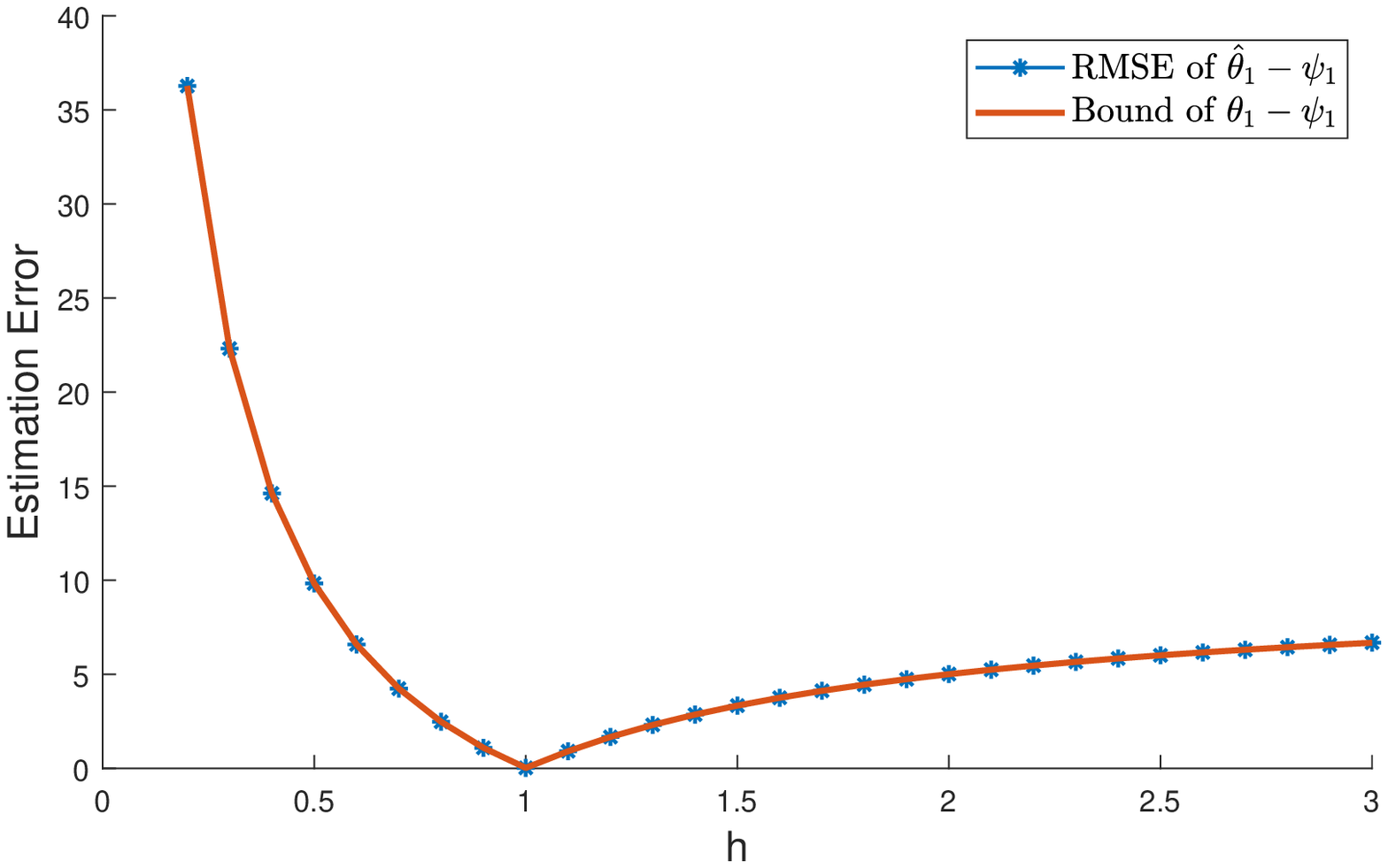}\label{fig:RMSEvsbiasedBound:different_h}
    \vspace{-0.3cm}
    \includegraphics[width=0.4\textwidth]{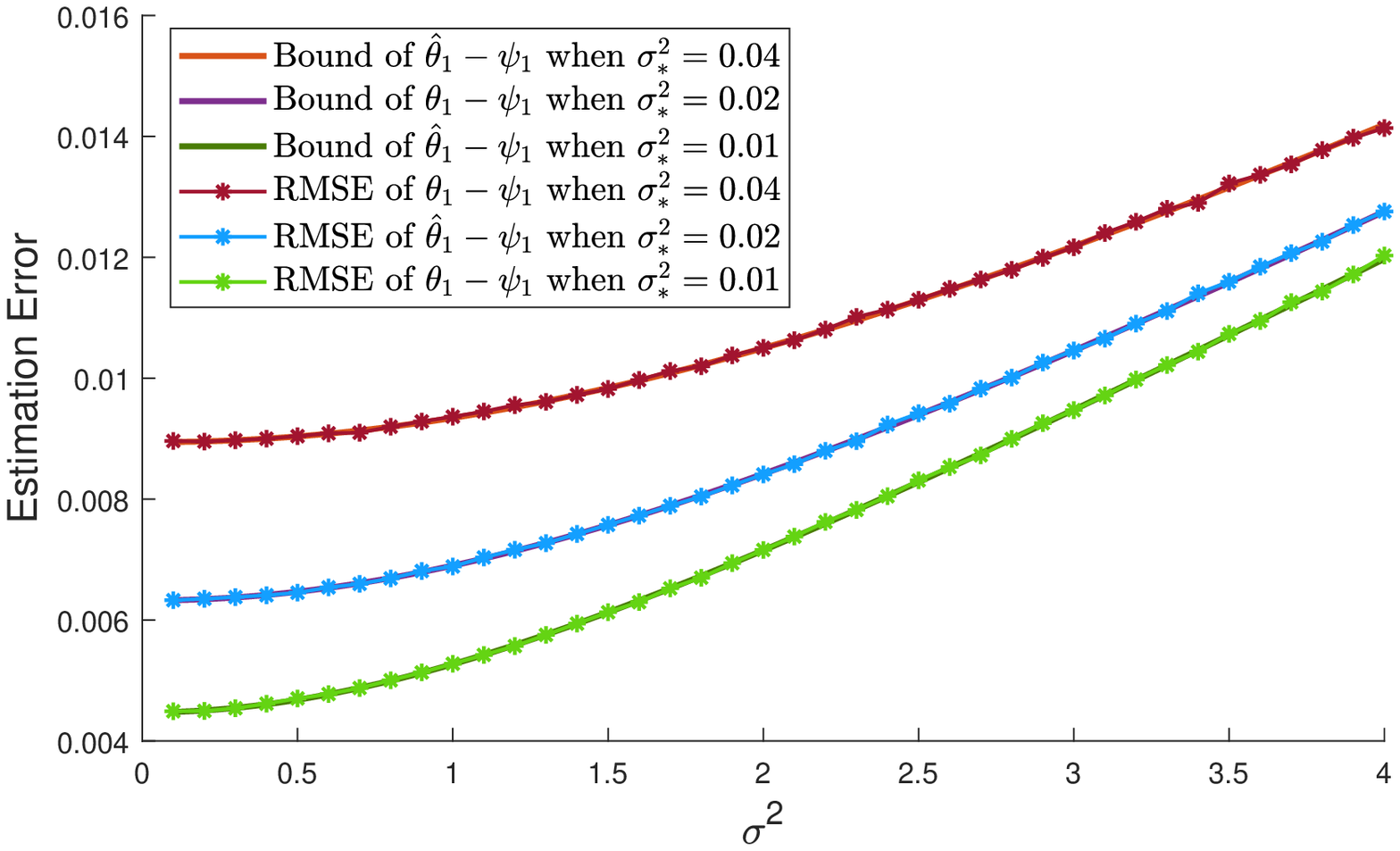}\label{fig:RMSEvsbiasedBound:different_Sigma}
    \caption{RMSE vs bound on $\hat{\btheta} - \bpsi$ under different model misspecifications: when varying $h$ (top panel) and varying $\sigma^2$ (bottom panel).}
    \label{fig:RMSEvsbiasedBound_different_level}
\end{figure}

\section{Conclusion}
\label{sec:conclusion}
Being able to predict the estimation error of an estimator is important for benchmarking and design purposes. This paper extends the works on CRB-type bounds for misspecified models to a general Bayesian setting where the parameters of true and assumed models are assumed random variables.
A new possible pseudotrue parameter definition is proposed, accounting for the Bayesian modeling of interest here, which is fundamental to derive the so-called MBCRB. The bound is particularized to the case of mismatched linear and Gaussian models, yielding to closed-form expressions that are simpler to compute and interpret, which can be useful in real-world (Bayesian) linear regression applications. In cases where true and assumed parameters belong to the same parameter space, we also provided an extended biased bound on the error between both which leverages the previous MBCRB result. 
Simulation results validates the proposed MBCRB for a number of misspecified situations.

\vfill\pagebreak

\bibliographystyle{IEEEtran}
\bibliography{strings,refs}

\end{document}